\documentclass{conm-p-l}

\newtheorem{theorem}{Theorem}[section]
\newtheorem{lemma}[theorem]{Lemma}
\newtheorem{proposition}[theorem]{Proposition}
\newtheorem{corollary}[theorem]{Corollary}

\theoremstyle{definition}

\theoremstyle{remark}
\newtheorem{remark}[theorem]{Remark}

\numberwithin{equation}{section}

\usepackage{amssymb}

\DeclareMathOperator{\Hom}{Hom}
\DeclareMathOperator{\Ind}{Ind}

\font \normalfrak eufm10

\def\fr#1{\hbox{\normalfrak #1}}

\def\sh{{\varPhi}}

\def\r{\text{reg}}
\def\tr{{\theta\text{-reg}}}
\def\tu{{\text{tu}}}
\def\G{{\bf G}}
\def\H{{\bf H}}
\def\T{{\bf T}}
\def\Q{{\mathbb Q}}
\def\wtp{{\widetilde\pi}}
\def\wtV{{\widetilde V}}
\def\wtw{{\tilde w}}
\def\wtv{{\tilde v}}
\def\wtW{{\widetilde W}}
\def\wtk{{\widetilde \kappa}}
\def\F{{\fr f}}
\def\U{{\mathcal U}}

\begin{document}

\title{Spherical characters: the supercuspidal case}

\author{Fiona Murnaghan}
\address{Department of Mathematics, University of Toronto, Toronto, 
Canada M5S 2E4}
\email{fiona@math.toronto.edu}
\thanks{The author's research is supported by an NSERC Discovery Grant}

\date{March 30, 2007}

\dedicatory{This paper is dedicated to the memory of George Mackey.}


\begin{abstract}
We exhibit a basis for the space of spherical
characters of a distinguished supercuspidal representation
$\pi$
of a connected reductive $p$-adic group,
subject to the assumption that $\pi$ is obtained
via induction from a representation of an open 
compact mod centre subgroup.
We derive an integral formula for each spherical
character belonging to the
the basis. This formula involves integration of
a particular kind of matrix coefficient of
$\pi$. We also obtain a similar formula for the
function realizing the spherical character.
In addition, 
we determine, subject to some conditions, which of these spherical
characters vanish identically on an open neighbourhood
of the identity. We verify that the requisite conditions
are always satisfied for distinguished tame supercuspidal 
representations
of groups that split over tamely ramified extensions.
\end{abstract}

\maketitle

\section{Introduction} The reader may refer to later sections
for specific information regarding definitions and notation.
Let $G$ be a connected reductive $p$-adic group. 

Let $\theta$ be an involution of $G$ and let
$H$ be a reductive $p$-adic group that is a subgroup
of the group of fixed points $G^\theta$  of $\theta$
 and contains
the identity component of $G^\theta$. Then $G/H$
is a reductive $p$-adic symmetric space.
Harmonic analysis on reductive $p$-adic symmetric spaces
involves the study of $H$-biinvariant
distributions on $G$.
Unlike harmonic analysis on real and complex reductive
symmetric spaces, which has been studied extensively,
there are many open questions in this area. 

Many results in harmonic analysis on $G$ have
some kind of analogue in harmonic analysis
on $G/H$. 
In \cite{RR}, Rader and Rallis carried over to reductive
$p$-adic symmetric spaces some results of Howe
and many results of Harish-Chandra (mainly from
the Queen's lectures (\cite{HC2})). 
The $H$-biinvariant distributions that play a role analogous to
that played in harmonic analysis on $G$ by
characters of irreducible admissible representations 
are the spherical characters of class one irreducible
admissible representations. (See
 Section~\ref{sec:sphchdef} for the definition of
spherical character.)

In this paper we study the basic properties of
spherical characters of class one irreducible
supercuspidal representations. Let $\pi$ be such
a representation.
In Theorem~\ref{spch}, we attach an 
$H$-biinvariant distribution $D_\varphi$ to each matrix 
coefficient $\varphi$ of $\pi$. 
If $\varphi$ does not 
have certain invariance properties,
then $D_\varphi=0$. 
Each $D_\varphi$ is
a spherical character of $\pi$.
Using results from \cite{RR} (as stated in Theorem~\ref{RRF}
of this paper), we obtain an integral formula for
each spherical character $D_\varphi$ as a function on
the $\theta$-regular set.
Our integral formulas for $D_\varphi$, as
a distribution and as a function on the $\theta$-regular
set, are analogues of results of Harish-Chandra 
that express
the ordinary character of an irreducible
supercuspidal representation in
terms of integration of a matrix coefficient
of the representation (see Theorem~\ref{scthm}).

Harish-Chandra's integral formulas for the ordinary 
characters of irreducible supercuspidal represntations 
have proved
useful both in computing character values and in
deriving qualitative properties of characters.
Our hope is that the formulas for spherical
characters that we obtain in this paper will
prove useful in a similar way. In related work,
currently in progress, these formulas are being used
in the study of germ expansions of spherical characters
of distinguished supercuspidal representations.

Assume that $\pi$ is a class one irreducible
supercuspidal representation that can
be realized as an induced representation (as
discussed in Section~\ref{sec:scind}).
A lemma of Hakim and Mao (\cite{HMa}) describing
the space of $H$-invariant linear functionals
on the space of $\pi$ is
used to prove that the space of spherical
characters of $\pi$ is spanned by the
distributions $D_\varphi$ as $\varphi$
ranges over matrix coefficients of $\pi$.
We exhibit a set of matrix coefficients
that give rise to  a set of linearly independent
spherical characters, thus obtaining
a basis for the space of spherical characters
of $\pi$ (see Theorem~\ref{dimen}).
We remark that certain of the examples discussed
in \cite{HM} are examples
of supercuspidal representations for which
the space of spherical characters has dimension
greater than one (see Remark~\ref{dimg}).

In Section~\ref{sec:vanish}, subject to some
conditions,
we determine which of
the spherical characters of $\pi$ vanish
identically on the intersection of an open 
neighbourhood of the identity with the
$\theta$-regular set. In particular,
if the dimension of $H$-invariant linear functionals
on the space of $\pi$ has dimension greater
than one, there exist spherical characters
of $\pi$ that exhibit this vanishing property.
Note that this contrasts with the behaviour of
ordinary characters of irreducible admissible
representations, which never vanish identically
on
any open neighbourhood of the identity.
At the end of the section,
we state results from \cite{HM} which
imply  that the conditions needed for
the vanishing results 
always hold 
for distinguished tame supercuspidal representations
of groups that split over tamely ramified extensions
of $F$.

\section{General notation and definitions} Let $G$ be
a reductive $p$-adic group. That is, $G$
is the group of $F$-rational points $\G(F)$
of a reductive algebraic $F$-group $\G$,
where $F$ is a nonarchimedean local field. 
Because some of the results that we use in this paper
have not been proved
in the positive characteristic setting, we will assume
that $F$ has characteristic zero.
Thus $F$ is a finite extension of $\Q_p$,
where $p$ is the residual characteristic of $F$.

We will assume that $\G$ is connected.
However, certain reductive subgroups of $\G$ that
appear in this paper may not be connected.

Let $\pi$ be a complex representation
of $G$. A vector in the space $V$ of $\pi$
is said to be smooth if it is fixed by an
open subgroup of $G$. A smooth representation 
$\pi$ is a representation having the property that
every vector in $V$ is smooth.
If $\pi$ is smooth and has the additional property
that the space of $K$-fixed vectors
in $V$ is finite-dimensional for every
compact open subgroup $K$ of $G$,
we say that $\pi$ is admissible.
An admissible representation $\pi$ is
supercuspidal whenever the matrix 
coefficients of $\pi$ are compactly supported
modulo the centre $Z$ of $G$.

Let $V^*$ be the dual space of $\pi$ and let
$\pi^*$ be the representation of $G$ that is dual
to $\pi$. Let $\pi$ be a smooth representation
of $G$. The dual representation $\pi^*$
is not necessarily smooth. 
The contragredient $\wtp$ of $\pi$ is
the restriction of $\pi^*$ to the subspace $\wtV$ of smooth
vectors in $V^*$, and is clearly a smooth representation. 
It is easy to see that $\pi$ is admissible, resp.\ supercuspidal,
 if and only if $\wtp$ is admissible, resp.\ supercuspidal.

The usual pairing on $\wtV\times V$ will be denoted
by $\langle\cdot,\cdot\rangle$.
We identify $V$ and $\widetilde{\wtV}$ via
the isomorphism that takes a vector $v$ in $V$
to the smooth linear functional 
$\tilde v\mapsto \langle \tilde v,v\rangle$ 
on $\wtV$.
The pairing $\langle \cdot,\cdot\rangle$ extends
in the obvious way to the union of $V^*\times V$
and $\wtV\times \wtV^*$. The extension
will also be denoted by $\langle \cdot,\cdot\rangle$.

Let $C_c^\infty(G)$ be the space of 
complex-valued, locally
constant, compactly supported functions on $G$.
A distribution on $G$ is a linear functional on $C_c^\infty(G)$.
A $G$-invariant distribution on $G$ is one
that takes the same value on $f$ and the function $g\mapsto
f(xgx^{-1})$ for all $f\in C_c^\infty(G)$ and
all $x\in G$.

If $\theta:\G\rightarrow \G$ 
is an automorphism of order two that
is defined over $F$, we
say that $\theta$ is an involution of $G$.
Given such an involution, let $\G^\theta$
be the group of fixed points of $\theta$
in $\G$, and let $(\G^\theta)^\circ$ be
the identity component of $\G^\theta$.
Let $\H$ be
an $F$-subgroup of $\G$ such that $(\G^\theta)^\circ
\subset\H\subset \G^\theta$. Let $H=\H(F)$.
The space $G/H$ is called a (reductive) $p$-adic 
symmetric space.
A distribution on $G$ is said to be $H$-biinvariant
if it takes the same value on $f$ and on the function
$g\mapsto f(h_1gh_2)$ for all $f\in C_c^\infty(G)$ and
all $h_1$ and $h_2\in H$.

Let $G^\r$ be the set of regular 
elements of $G$. An element $g$ of $G$ belongs
to $G^\r$ if and only if the identity component
of the centralizer of $g$ in $\G$ is a maximal
torus. The set $G^\r$ is open and dense in $G$.

A torus $\T$ in $\G$ is called $\theta$-split
if $\theta(t)=t^{-1}$ for all $t\in \T$.
Let $\theta$ be an involution of $G$. The set
$G^{\tr}$ of $\theta$-regular elements in
$G$ consists of the elements $g$ in $G$
having the property that the intersection
of the centralizer of $g\,\theta(g)^{-1}$ in $\G$
with the connected component of the identity
in $\{\, x\in G\ | \ \theta(x)=x^{-1}\,\}$
is a maximal $\theta$-split torus in $\G$.
The set $G^\tr$ is open and dense in $G$.

Haar measure on a unimodular locally compact group
 $G$ will be denoted by $dg$.
If $G_1$ is a closed unimodular subgroup of
$G$, then $dg^\times$ will denote a 
$G$-invariant measure on the coset space $G/G_1$.

\section{Ordinary characters}

Let $\pi$ be an admissible representation of $G$.
If $f\in C_c^\infty(G)$, then the operator $\pi(f)=\int_G f(g)\,\pi(g)
\,dg$ has finite rank. 
The character $\Theta_\pi$ of $\pi$ is the $G$-invariant distribution
defined by $\Theta_\pi(f)={\rm trace}\,\pi(f)$, $f\in C_c^\infty(G)$.

The first part of the following theorem was originally proved 
by Harish-Chandra
(\cite{HC2}) for $G$ connected. Although we will not
need it here, we remark that it has since been
generalized to disconnected $G$ by Clozel (\cite{C}).
The second part of the theorem was proved by Harish-Chandra
for $K$ a good maximal compact subgroup, and was
generalized to arbitrary open compact subgroups 
by Rader and Silberger (\cite{RS}).

\begin{theorem}\label{chthm}
 Let $\pi$ be an admissible finite-length representation
of $G$.
\begin{enumerate}
\item The character distribution $\Theta_\pi$ is given by
integration against a locally integrable function
(also denoted by $\Theta_\pi$ and called the character
of $\pi$) on $G$. The function $\Theta_\pi$ is locally
constant on $G^\r$.
\item Let $K$ be an open compact subgroup of $G$.
Then $\int_K \pi(kgk^{-1})\, dk$ has finite rank
 whenever $g\in G^\r$.
If Haar measure on $K$ is normalized so that
$K$ has volume one, then
$$
\Theta_\pi(g)=\text{trace}\left(\int_K \pi(kgk^{-1})\, dk\right),
                         \qquad g\in G^\r.
$$
\end{enumerate}
\end{theorem}

\begin{theorem}\label{scthm}(\cite{HC1})
Let $\pi$ be an irreducible supercuspidal representation of
$G$ and let $d(\pi)$ be the formal degree of $\pi$.
If $\varphi$ is a matrix coefficient of $\pi$, then
\begin{enumerate}
\item $$
\varphi(1)\,\Theta_\pi(g)= d(\pi)\int_{G/Z}\int_K
\varphi(xkg k^{-1}x^{-1})\, dk\, dx^\times,
\qquad g\in G^{\r}.
$$
\item $$\varphi(1)\Theta_\pi(f)=d(\pi)\, \int_{G/Z}\int_G
f(g)\varphi(xg x^{-1})\, dg\, dx^\times,
\qquad f\in C_c^\infty(G).
$$
\end{enumerate}
\end{theorem}

\begin{remark} The above theorem was generalized
to discrete series characters by Rader and
Silberger (\cite{RS}).
\end{remark}

\section{Distinguished representations and $p$-adic 
symmetric spaces}
\label{sec:dist}

Let $\pi$ be a smooth representation of $G$.
 Let $\Hom_H(\pi,1)$ be the space
of $H$-invariant elements in the dual $V^*$
of the space $V$ of $\pi$.
The representation $\pi$ is said to be distinguished
(or $H$-distinguished) if $\Hom_H(\pi,1)$ is nonzero.

Irreducible admissible representations $\pi$ having the
property that
both $\pi$ and $\wtp$ are distinguished
are referred to as
as class one representations.
The supercuspidal representations that we study in
this paper have the property that $\Hom_H(\pi,1)$ is
isomorphic to $\Hom_H(\wtp,1)$ (see Remark~\ref{sphd}).
Hence such representations
are distinguished if and only if they are class one.

Information on the structure of reductive $p$-adic
symmetric spaces
may be found in the papers \cite{HH}, \cite{HW}
and \cite{RR}.
It is customary to identify $H$-invariant distributions
on $G/H$ with $H$-biinvariant distributions on $G$
via composition with the map $f\mapsto {\bar f}$ where
${\bar f}(gH)=\int_H f(gh)\, dh$, $f\in C_c^\infty(G)$,
$g\in G$.

The spherical characters of class 
one representations 
play a key role in harmonic analysis on $G/H$.
Because the dimension of $\Hom_H(\pi,1)$ can be greater
than one,
the dimension of the space of spherical characters
of a representation $\pi$ can be greater than one.
Some examples of supercuspidal representations
for which $\dim\Hom_H(\pi,1)$ exceeds one
are given in \cite{HM}.

The reader may consult the introduction of \cite{RR} for a
discussion of the connections between harmonic analysis on
$p$-adic symmetric spaces and the relative trace formula,
including nonvanishing of period integrals of automorphic forms.

A study of $\Hom_H(\pi,1)$ for representations $\pi$ arising
via parabolic induction has been carried
out by Blanc and Delorme (\cite{BD}).

Distinguished supercuspidal representations have
been studied by various researchers. The papers 
\cite{H1}, \cite{H2}, \cite{H3}, \cite{HMa}, \cite{HM1},
\cite{HM2}, \cite{HM}, \cite{Pr1} and \cite{Pr2} (as well
as others not mentioned here) study
distinguished supercuspidal representations using approaches
that involve realizing supercuspidal
representations as induced representations (in the
sense discussed in Section~\ref{sec:scind}).
The most general cases are treated in \cite{HM},
which makes a detailed study of $\Hom_H(\pi,1)$ for
the supercuspidal representations constructed by
Yu (\cite{Y}).

Some authors (see, for example, \cite{AKT} and \cite{Ka}) 
show that in certain cases existence of
poles of $L$-functions can determine whether 
representations are distinguished.
This is related to the fact that in some
contexts the set of distinguished representations
can be described (sometimes conjecturally)
 as the set of representations
arising via some sort of functorial lift.
See, for example, \cite{AR}, \cite{AT}, \cite{F}
and \cite{HM1}.

The papers \cite{H1}, \cite{H2} and
\cite{H3} contain results on spherical
characters for particular examples.

We remark that, because of our focus on the $p$-adic case,
we are not mentioning papers that treat
distinguished automorphic representations exclusively.

\section{Spherical characters - definition and basic properties}
\label{sec:sphchdef}

Let $\pi$ be an irreducible admissible representation of $G$
and let $V$ be the space of $\pi$.
Let $\lambda_{\wtp}\in \wtV^*$. If $f\in C_c^\infty(G)$, define
$\pi(f)\lambda_{\wtp}\in \wtV^*$ by
$$
\langle\, \tilde v,\pi(f)\lambda_{\wtp} \,\rangle=
\langle\,\wtp(\check f)\tilde v,\lambda_{\wtp}\,\rangle,
\qquad \tilde v\in \wtV,
$$
where $\check f\in C_c^\infty(G)$ is defined by
${\check f}(g)=f(g^{-1})$, $g\in G$.
It is a simple matter to check that
$\pi(f)\lambda_{\wtp}$ is smooth for every $f\in C_c^\infty(G)$, 
that is, $\pi(f)\lambda_{\wtp}
\in \widetilde{\wtV}$. 
Hence, since we are identifying $V$ and $\widetilde{\wtV}$,
 we may (and do) identify $\pi(f)\lambda_{\wtp}$ 
with an element of $V$.
Let $\lambda_\pi\in V^*$. The map
$$ 
f\mapsto \langle\, \lambda_\pi, \pi(f)\lambda_{\wtp}\,\rangle,
\qquad f\in C_c^\infty(G),
$$
defines a distribution on $G$. These types of distributions can be
viewed as generalizations of those distributions given by
integration against matrix coefficients of representations.

If $\pi$ is a class one representation, the
spherical characters of $\pi$ are the linear combinations
of distributions of the above form
that are $H$-biinvariant. 
Let $\lambda_\pi$ and $\lambda_\wtp$
be nonzero elements of $\Hom_H(\pi,1)$ and $\Hom_H(\wtp,1)$,
respectively. Set
$$
\sh_\pi(f)=\langle\,\lambda_\pi, \pi(f)\lambda_{\wtp}\,\rangle,
\qquad f\in C_c^\infty(G).
$$
It follows from $H$-invariance of $\lambda_\pi$ and
$\lambda_\wtp$ that $\sh_\pi$ is $H$-biinvariant.
We will refer to $\sh_\pi$ as the spherical
character of $\pi$ associated to the pair
of linear functionals $\lambda_\pi$ and
$\lambda_{\wtp}$.
 
The statements in the second part of the following 
theorem can be viewed
as analogues for spherical characters of the results for ordinary
characters stated in Theorem~\ref{chthm}.
However, as indicated in \cite{RR}, the function that
realizes the spherical character is not generally locally 
integrable on $G$.
Let $C_c^\infty(G^\tr)$ be the space of complex-valued,
compactly supported, locally constant functions on
$G^\tr$.

\begin{theorem}\label{RRF}
 (\cite{RR})  Let $\sh_\pi$ be the spherical
character of $\pi$ that is associated to elements 
$\lambda_\pi\in \Hom_H(\pi,1)$
and $\lambda_\wtp\in \Hom_H(\wtp,1)$. Let $K$ be a compact
open subgroup of $G$.
\begin{enumerate}
\item If $g\in G^{\tr}$, then 
$\int_{K\cap H} \pi^*(kg^{-1})\lambda_\pi\, dk$
lies in $\wtV$. The function $g\mapsto \int_{K\cap H} 
\pi^*(kg^{-1})\lambda_\pi\, dk$ is a $C^\infty$ function from
$G^\tr$ to $\wtV$. 

\item The restriction of $\sh_\pi$ to $C_c^\infty(G^\tr)$ 
is given
by integration against a locally constant function
(also denoted by $\sh_\pi$) on $G^\tr$. If
the measure $dk$ is normalized so that $K\cap H$
has volume one, then
$$
\sh_\pi(g)=\big\langle\, \int_{K\cap H} \pi^*(kg^{-1})\, 
\lambda_\pi\, dk\, , \,
\lambda_\wtp\, \big\rangle, \qquad g\in G^\tr.
$$
\end{enumerate}
\end{theorem}

\section{Matrix coefficients and spherical characters 
of supercuspidal representations}\label{sec:sphform}

Throughout this section we assume that  $\pi$ is an 
irreducible supercuspidal representation
having the property that the central quasicharacter
$\chi_\pi$ of $\pi$ is trivial on $H\cap Z$.
Note that $\pi$ cannot be distinguished if it
does not have this property.

The restriction of
a matrix coefficient of $\pi$ to $H$ is
$H\cap Z$-biinvariant and compactly
supported modulo $H\cap Z$. Let $V$ be
the space of $\pi$. Given 
$\wtv\in V$, define $\lambda_\wtv
\in V^*$ by
$$
\langle\, \lambda_\wtv,\, v\,\rangle = 
\int_{H/H\cap Z} \langle\, \wtv,\, \pi(h)v\,\rangle\, dh^\times,
\qquad v\in V.
$$
Note that
$\lambda_\wtv\in \Hom_H(\pi,1)$. Later in the paper
we describe which vectors $\wtv$ give rise
to nonzero elements of $\Hom_H(\pi,1)$ (when
$\pi$ is an induced representation).
Similarly, if we fix $v\in V$, we can
define $\lambda_v\in \Hom_H(\wtp,1)$ by
$$
\langle\, \wtv,\, \lambda_v\,\rangle = 
\int_{H/H\cap Z} \langle\, \wtp(h)\wtv,\, v\,\rangle\, dh^\times,
\qquad \wtv\in \wtV.
$$

Fix $v_0\in V$ and $\wtv_0\in \wtV$. Let $\varphi(g)
=\langle\, \wtv_0,\pi(g)v_0\,\rangle$, $g\in G$.
As we show below, we may use the matrix coefficient
$\varphi$ to define an $H$-biinvariant distribution
on $G$, and this distribution is a spherical character
of $\pi$ which is nonzero if and only if
$\lambda_{\wtv_0}$ and $\lambda_{v_0}$ are
both nonzero. Furthermore, the formula of
Rader and Rallis (Theorem~\ref{RRF}(2)) converts into
an integral formula for this spherical
character (see (4) below). The latter formula can 
be viewed as an 
analogue for spherical characters 
of Harish-Chandra's integral formula for the 
ordinary character of a supercuspidal 
representation (see Theorem~\ref{scthm}(1)).
Similarly, the expression for the distribution
$\sh_\pi$ in terms of integration of a matrix
coefficient 
is an analogue of the one in Theorem~\ref{scthm}(2).

\begin{theorem}\label{spch} Let $\pi$ and $\varphi$ be as above.
\begin{enumerate}
\item The map $f\mapsto D_{\varphi}(f)=
\int_{H/H\cap Z}\int_{H/H\cap Z} \int_G f(g)\, 
\varphi(h_1gh_2)\, dg\, dh_1^\times dh_2^\times$,
$f\in C_c^\infty(G)$, 
defines an $H$-biinvariant distribution on
$G$.
\item Fix $v_1\in V$ and $\wtv_1\in \wtV$.
Let $\dot\varphi(g)=\langle\, \wtv_1,\,\pi(g)^{-1}v_1\,\rangle$,
$g\in G$.
Suppose that
$f\in C_c^\infty(G)$ has the property that
$\dot\varphi(g)=\int_Z f(gz)\chi_\pi(z)\,dz$ for
all $g\in G$. 
Then
$$
D_\varphi(f)= d(\pi)^{-1} \langle\, \lambda_{\wtv_0},\, v_1\,\rangle
\langle \, \wtv_1,\, \lambda_{v_0}\,\rangle.
$$
\item Let $\sh_\pi$ be the spherical character of
of $\pi$ associated to $\lambda_{\wtv_0}$ and
$\lambda_{v_0}$. Then $\sh_\pi=D_\varphi$.
Moreover, $\sh_\pi$ is nonzero if and only if
$\lambda_{v_0}$ and $\lambda_{\wtv_0}$ are nonzero.
\item
Let $K$ be a compact open subgroup of $G$.
Normalize Haar measure on $K\cap H$ so
that $K\cap H$ has volume one. Then
$$
\sh_\pi(g)=\int_{H/H\cap Z}\int_{K\cap H} \int_{H/H\cap Z}
\varphi(h_2gkh_1)\, dh_2^\times\,dk\,dh_1^\times,
\qquad g\in G^\tr.
$$
\end{enumerate}
\end{theorem}
\vfil\eject

\begin{remark}\begin{enumerate}
\item If we combine (2) and the first
part of (3), we obtain a special case of Lemma~4
of \cite{H4}.
\item Clearly we may extend the definition
of $D_\varphi$ to any function $\varphi$
belonging to the span of the matrix coefficients
of $\pi$. For each such function $\varphi$,
 $D_\varphi$ belongs to the space of spherical 
characters of $\pi$.
\end{enumerate}
\end{remark}

\begin{proof} Fix $f\in C_c^\infty(G)$. Define $\psi(h_1,h_2)= 
\langle\, \wtv_0,\, \pi(h_1)\pi(f)\pi(h_2)v_0\,\rangle$,
$h_1$, $h_2\in H$.
Note that $\psi$ is $H\cap Z$-biinvariant.
Since $\pi(f)$ has finite rank, the span of
the vectors $\pi(f)\pi(h_2)v_0$,
as $h_2$ ranges over $H$, is finite-dimensional.
Thus there is a finite set of matrix coefficients
of $\pi$ such the the span of their restrictions
to $H$ contains all of the functions
$h_1\mapsto \psi(h_1,h_2)$, as $h_2$ varies
in $H$.
It follows that there exists
a compact subset $C_1$ of $H$ such that
the support of each of the functions $h_1\mapsto \psi(h_1,h_2)$
lies inside $C_1(H\cap Z)\times H$.

Since $\psi(h_1,h_2)=\langle\, \wtp({\check f})\wtp(h_1)\wtv_0,
\, \pi(h_2)v_0\,\rangle$, we may use the fact that
$\wtp({\check f})$ has finite rank to see that
there exists a compact subset $C_2$ of $H$ such
that the support of each function $h_2\mapsto \psi(h_1,h_2)$
lies inside $C_2(H\cap Z)$.

It now follows that $\psi$ is supported in $C_1(H\cap Z)
\times C_2(H\cap Z)$. That is, $\psi$ has compact
support modulo $(H\cap Z)\times (H\cap Z)$.
Therefore the integral 
$$
\int_{H/H\cap Z}\int_{H/H\cap Z}\psi(h_1,h_2)
\, dh_1^\times\, dh_2^\times$$ converges. 
Because $\psi(h_1,h_2)=\int_G f(g)\,\varphi(h_1gh_2)\, dg$,
the above integral is equal to $D_\varphi(f)$.
It is clear from the form of $D_\varphi(f)$
that $D_\varphi$ is $H$-biinvariant. 

Next, fix $f\in C_c^\infty(G)$ as in (2). Let $h_1$, $h_2\in H$.
Then
\begin{equation*}
\begin{split}
\int_G f(g)\, \varphi(h_1 g h_2)\, dg &=\int_{G/Z}\int_Z
\, f(gz)\, \varphi(h_1gzh_2)\, dz\, dg^\times
=\int_{G/Z} \varphi(h_1gh_2)\, \dot\varphi(g)\, dg^\times \cr
&=\int_{G/Z} \langle\, \wtp(h_1^{-1})\wtv_0,\, \pi(g)\pi(h_2)v_0\,
\rangle\langle\, \wtv_1,\, \pi(g^{-1})v_1\,\rangle\, dg^\times \cr
& =d(\pi)^{-1} \langle\, \wtp(h_1^{-1})\wtv_0,\, v_1\,\rangle
\langle\wtv_1,\,\pi(h_2)\, v_0\,\rangle.
\end{split}
\end{equation*}
Note that the fourth equality above is obtained
via an application of the orthogonality relations
for matrix coefficients of quasi-discrete series 
representations.
In view of the above, it now follows from the definitions
of $D_\varphi$, $\lambda_{\wtv_0}$ and $\lambda_{v_0}$
that (for this particular choice of $f$) 
$D_\varphi(f)$ has the form given in (2).

Now we fix an arbitrary $f\in C_c^\infty(G)$. Then
\begin{equation*}
\begin{split}
\sh_\pi(f) &= \langle \, \lambda_{\wtv_0},\, \pi(f)\lambda_{v_0}\,\rangle
=\int_{H/H\cap Z} \langle\, \wtv_0,\pi(h_1)\pi(f)
                      \lambda_{v_0}\,\rangle\,dh_1^\times\cr
&=\int_{H/H\cap Z} \langle\, \wtp({\check f})\pi(h_1^{-1})\wtv_0,
\, \lambda_{v_0}\,\rangle\, dh_1^\times \cr
&=\int_{H/H\cap Z}\int_{H/H\cap Z} \langle\, \wtp({\check f})\pi(h_1^{-1})\wtv_0,
\, \pi(h_2)v_0\,\rangle\, dh_1^\times\, dh_2^\times =D_\varphi(f)
\end{split}
\end{equation*}
The assertion about nonvanishing of $\sh_\pi$ is an immediate consequence
of $\sh_\pi=D_\varphi$ and (2).

Fix $g\in G^\tr$. Let $K$ be as in the statement of the theorem.
Applying Theorem~\ref{RRF}(2) and using the definitions of
$\lambda_{\wtv_0}$ and $\lambda_{v_0}$, we find that
\begin{equation*}
\begin{split}
\sh_\pi(g) & =\big\langle \, \int_{K\cap H} \pi^*(kg^{-1})
 \lambda_{\wtv_0}\, dk,\, \lambda_{v_0}\,\big\rangle\cr
&= \int_{H/H\cap Z}\big\langle\, \int_{K\cap H} \pi^*(kg^{-1})
\lambda_{\wtv_0}\,dk,\, \pi(h_1)v_0\,\big\rangle\, dh_1^\times\cr
&=\int_{H/H\cap Z}\int_{K\cap H} \langle\, \lambda_{\wtv_0},
\,\pi(gk^{-1}h_1)v_0\,\rangle\, dk\, dh_1^\times\cr
&=\int_{H/H\cap Z}\int_{K\cap H}\int_{H/H\cap Z}
\langle\, \wtv_0,\, \pi(h_2gk^{-1}h_1)v_0\,\rangle dh_2^\times
dk\, dh_1^\times.
\end{split}
\end{equation*}
Upon making the change of variables $k\mapsto k^{-1}$, we
obtain the desired expression for $\sh_\pi(g)$.
\end{proof}

\section{Supercuspidal representations as induced representations}
\label{sec:scind}

Before discussing spherical characters of
supercuspidal representations in more detail, we make some remarks
about irreducible supercuspidal representations that can be
realized as induced representations.

Let $J$ be an open subgroup of $G$ that contains $Z$ and
has the property that $J/Z$ is compact.
Let $\kappa$ be an irreducible smooth representation of
$J$. Because $J/Z$ is compact, $\kappa$ is finite-dimensional.
We denote the space of $\kappa$ by $W$.
Let $\pi=\Ind_J^G\kappa$ be the smooth representation
of $G$ that is obtained via compact induction from
the representation $\kappa$. The space $V$ of
$\pi$ is the set of functions $\F$ from $G$ to $W$
satisfying
\begin{itemize}
\item $\F(jg)=\kappa(j)\F(g)$ for all $j\in J$ and $g\in G$
\item The support of $\F$ lies inside a finite union of
right cosets of $J$ in $G$.
\end{itemize}

\noindent If $\F\in V$ and $g\in G$, then $(\pi(g)\F)(x)=\F(xg)$
for all $x\in G$.

If $g\in G$, let $J^g=J\cap gJg^{-1}$. Define a representation
$\kappa^g$ of $J^g$ by $\kappa^g(x)=\kappa(g^{-1}xg)$, 
$x\in J^g$. It is known that $\pi$ is irreducible
if and only if for every $g\in G-J$ the representations
$\kappa^g$ and $\kappa\,|\, J^g$ have no constituents
in common. We remark that this irreducibility criterion takes
the same form as Mackey's irreducibility criterion for
induced representations of finite groups.

Any smooth irreducible representation is admissible.
Observe that $\pi$ has matrix coefficients that are compactly
supported modulo $Z$. Hence, whenever $\pi$ is irreducible,
$\pi$ is supercuspidal.

It is conjectured that every irreducible supercuspidal
representation of a connected reductive $p$-adic group $G$ 
has the above form (for some choice of $J$ and $\kappa$).
This conjecture has been verified for many groups.
We do not take the time for a detailed discussion of
cases for which
the conjecture has been verified.
We remark that if $G$ is a general linear group, the conjecture 
follows from work on Bushnell and Kutzko (\cite{BK}) on
parametrizing the admissible dual of $G$. Also, if $G$
is a connected reductive $p$-adic group that
splits over a tamely ramified extension of $F$ and satisfies 
some tameness hypotheses, J.-L.\ Kim (\cite{K}) has
shown that the tame supercuspidal representations
of $G$ constructed by Yu (\cite{Y}) exhaust the
irreducible supercuspidal representations of $G$.

\section{Spherical characters of induced supercuspidal
representations}

Suppose that $\pi$ is an irreducible
supercuspidal representation of $G$ that is of the form
$\Ind_J^G\kappa$ (as discussed in the previous section).
In the first part of this section, we state a 
result due to Hakim and Mao
that gives a direct sum decomposition of the space $\Hom_H(\pi,1)$.
Let $W$ and $\wtW$ be the spaces of $\kappa$ and of $\wtk$,
respectively.
If $J_1$ is a subgroup of $J$, we denote the spaces
of $J_1$-fixed vectors in $W$ and $\wtW$ by
$W^{J_1}$ and $\wtW^{J_1}$, respectively.
Each of the summands in the decomposition of
$\Hom_H(\pi,1)$ is isomorphic to
$\wtW^{J\cap gHg^{-1}}$ for some $g\in G$.
As indicated below, this implies
that $\Hom_H(\pi,1)$ is spanned by elements of the form
$\lambda_{\wtv_0}$ (as defined in Section~\ref{sec:sphform}) 
for particular kinds of vectors $\wtv_0\in \wtV$.
As a consequence, each spherical character
of $\pi$ is of the form $D_\varphi$ for some
finite linear combination $\varphi$ of matrix 
coefficients of $\pi$. 

At the end of
the section, we demonstrate that certain
spherical characters of $\pi$ are linearly
independent. This is used to 
show that the dimension
of the space of spherical characters of
$\pi$ is the square of the sum of the dimensions
of the spaces $W^{J\cap gHg^{-1}}$ as $g$
ranges over a 
set of representatives for
the $J$-$H$ double cosets in $G$.

The contragredient representation
$\wtp$ may (and will) be realized as $\Ind_J^G\wtk$.
Hence $\wtV$ is realized as the set of functions $\F$ from
$G$ to $\wtW$ that  have the same support properties
as functions in $V$, and satisfy
$\F(jg)=\wtk(j)\,\F(g)$ for all $j\in J$ and $g\in G$.

 Let 
$\langle\cdot,\cdot\rangle_W$ be the usual pairing
of $\wtW$ and $W$.
We normalize the invariant measure on $G/J$ in such a way
that the pairing $\langle \cdot,\cdot\rangle$ of $\wtV$ with $V$
is given by
$$
\langle\, \tilde\F,\F\,\rangle = \int_{G/J} \langle
\, \tilde\F(g),\F(g)\,\rangle_W\, dg^\times,
\qquad \tilde\F\in \wtV,\, \F\in V.
$$

\begin{lemma}\label{hakmao} (\cite{HMa}) 
Then
$$
\Hom_H(\pi,1)\simeq \bigoplus_{g\in J\backslash G/H}\,
\wtW^{J\cap gHg^{-1}},
$$
where the sum is over a set of representatives for
the $J$-$H$ double cosets in $G$.
\end{lemma}

\begin{remark}\label{sphd} As $\widetilde\wtW\simeq W$
and $W^{J\cap gHg^{-1}}$ is isomorphic to $\wtW^{J\cap gHg^{-1}}$
for all $g\in G$, it follows from the above lemma
that $\Hom_H(\pi,1)\simeq\Hom_H(\wtp,1)$. In particular,
$\pi$ is class one if and only if $\pi$ is distinguished.
\end{remark}

Suppose $g\in G$ is such that $W^{J\cap gHg^{-1}}\not=0$.
Let $\wtw$ be a nonzero element of $\wtW^{J\cap gHg^{-1}}$.
The element $\lambda_\pi\in \Hom_H(\pi,1)$ that corresponds
to $\wtw$ under the isomorphism of Lemma~\ref{hakmao}
is defined by
$$
\langle\,\lambda_\pi,\F \,\rangle=\int_{H/H\cap Z}
\langle\, \wtw, \F(gh)\, \rangle_W\, dh^\times,
\qquad \F \in V.
$$
We remark that the formula given in \cite{HMa} involves
an integral over $(H\cap g^{-1}Jg)\backslash H$.
As $g^{-1}Jg/Z$ is compact, the group $H\cap gJg^{-1}$
is compact modulo $H\cap Z$. Thus, as long as we normalize
measures appropriately, the above formula agrees with
the one from \cite{HMa}.

Let $\F_{\wtw}$ be the unique element of $\wtV$ 
that is supported on $J$ and satisifes $\F_{\wtw}(1)=\wtw$.
It follows from
the definition of $\lambda_\pi$ that 
$\lambda_\pi=\lambda_{\wtv_0}$ for $\wtv_0=\wtp(g^{-1})\F_\wtw$.

Similarly, if we fix a nonzero $w\in W^{J\cap g^\prime Hg^{\prime -1}}$
the associated $\lambda_{\wtp}\in \Hom_H(\wtp,1)$ is
of the form $\lambda_{v_0}$ with $v_0=\pi(g^{\prime -1})\F_w$,
where $\F_w$ is the function in $V$ that is supported on
$J$ and satisfies $\F_w(1)=w$.

In view of the above discussion,
the space of spherical characters of $\pi$
may be described as follows.

\begin{lemma}\label{mclemma}
Each spherical character of $\pi$ 
is of the form $D_\varphi$ for some function
$\varphi$ that is a linear combination of
matrix coefficients of $\pi$. Each matrix
coefficient in the linear combination
can be taken to be of the form 
$x\mapsto \langle\, \wtv_0,\, \pi(x)\,v_0\,
\rangle$ where $\wtv_0=\wtp(g^{-1})\F_\wtw$,
$g\in G$, $\wtw\in \wtW^{J\cap gHg^{-1}}$,
and $v_0=\pi(g^{\prime -1})\F_w$, $g^\prime
\in G$, and $w\in W^{J\cap g^\prime H g^{\prime -1}}$.
\end{lemma}

Let $\{\, g_i\ | \ i\in I\,\}$ be a 
set of representatives for a set of
distinct $J$-$H$ double cosets in $G$ for
which $W^{J\cap g_iHg_i^{-1}}\not=0$, $i\in I$.
For each $i\in I$, let $J_i=J\cap g_iHg_i^{-1}$ 
and $n_i=\dim W^{J_i}$.
Choose a basis $\beta_i$ of $W$
in such a way that the first $n_i$ vectors $w_1^{(i)},\dots,
w_{n_i}^{(i)}$ of $\beta_i$ form a basis of $W^{J_i}$,
and the first $n_i$ vectors $w_1^{(i)*},\dots
w_{n_i}^{(i)*}$ of the basis of $\wtW$ dual
to $\beta_i$ form a basis of $\wtW^{J_i}$.

Let $I^\prime=\{\, \alpha = (i,j)\ | \ i\in I,\  1\le j\le n_i\,\}$.
If $\alpha=(i,j)\in I^\prime$, set $v_\alpha= \pi(g_i^{-1})\F_{w_j^{(i)}}$
and $\wtv_\alpha=\wtp(g_i^{-1})\F_{\wtw_j^{(i)*}}$.
If $\alpha$, $\beta\in I^\prime$, let $\sh_\pi^{(\alpha,\beta)}$
be the spherical character associated to $\lambda_{\wtv_\alpha}$
and $\lambda_{v_\beta}$.

\begin{lemma} \label{indep}
The distributions $\sh_\pi^{(\alpha,\beta)}$,
$\alpha$, $\beta\in I^\prime$, are linearly independent.
\end{lemma}

\begin{proof} Let $\alpha$, $\beta\in I^\prime$.
According to Theorem~\ref{spch}(3),
 $\sh_\pi^{(\alpha,\beta)}=D_{\varphi^{(\alpha,\beta)}}$
where $\varphi^{(\alpha,\beta)}(g)=
\langle\, \wtv_\alpha,\, \pi(g)\, v_\beta\,\rangle$.
Choose $f_{(\alpha,\beta)}\in C_c^\infty(G)$ such
that
$$
\int_Z f_{(\alpha,\beta)}(gz)\,\chi_\pi(z)\, dz =\langle
\,\wtv_\alpha,\, \pi(g^{-1})\, v_\beta\,\rangle=
\varphi^{(\alpha,\beta)}(g^{-1}),\qquad g\in G.
$$

Let $\gamma$, $\delta\in I^\prime$. Applying Theorem~\ref{spch}(2),
we have
$$
\sh_\pi^{(\alpha,\beta)}(f_{(\gamma,\delta)})=d(\pi)^{-1}
\langle\, \lambda_{\wtv_\alpha},\, v_\delta\,\rangle\langle\, \wtv_\gamma,
\, \lambda_{v_\beta}\,\rangle.
$$
Suppose that $\alpha=(i,j)$ and $\delta=(\ell,m)$.
Then
\begin{equation*}
\begin{split}
\langle\, \lambda_{\wtv_\alpha},\, v_\delta\,\rangle &=
\int_{H/H\cap Z} \langle\, \wtp(g_i^{-1})\, \F_{w_j^{(i)*}},\,
\pi(hg_\ell^{-1})\, \F_{w_m^{(\ell)}}\,\rangle\, dh^\times\cr
&=\int_{H/H\cap Z} \langle\, w_j^{(i)*},\, \F_{w_m^{(\ell)}}
  (g_ihg_\ell^{-1})\, \rangle_Wdh^\times.
\end{split}
\end{equation*}
If $h\in H$, then $g_ihg_\ell^{-1}\in J$ implies $g_i\in Jg_\ell H$.
Thus $\langle\, \lambda_{\wtv_\alpha},\, v_\delta\,\rangle=0$ if $i\not=\ell$.
If $i=\ell$, then $\F_{w_m^{(i)}}(g_ihg_i^{-1})
\not=0$ if and only if $h\in g_i^{-1}Jg_i\cap H$. For such
$h$, $\F_{w_m^{(i)}}(g_i h g_i^{-1})=w_m^{(i)}$, since
$w_m^{(i)}\in W^{J_i}$.
Let $c_i$ be the volume of $(g_i^{-1}Jg_i\cap H)/Z$
relative to the measure $dh^\times$. When $i=\ell$, 
$$
\langle\, \lambda_{\wtv_\alpha},\, v_\delta\,\rangle=c_i\, 
\langle \, w_j^{(i)*},\, w_m^{(i)}\,\rangle,
$$
which, by definition of $w_j^{(i)*}$ and $w_m^{(i)}$,
is equal to $c_i$ when $m=j$ and zero when $m\not=j$.
We conclude that $\langle\, \lambda_{\wtv_\alpha},\, v_\delta\,\rangle$
is nonzero if and only if $\alpha=\delta$.

A similar argument shows that 
$\langle\, \wtv_\gamma,\, \lambda_{v_\beta}\,\rangle$ is 
nonzero if and only if $\gamma=\beta$.

Thus $\sh_\pi^{(\alpha,\beta)}(f_{(\gamma,\delta)})$
is nonzero if and only if $\alpha=\gamma$ and
$\beta=\delta$. This proves the lemma.
\end{proof}

The following theorem is a consequence of
the above lemma.

\begin{theorem}\label{dimen}
\begin{enumerate}
\item Let $g_1$, $g_2,\dots $ be representatives for the
set of all distinct $J$-$H$ double cosets in $G$ that
contain elements $g\in G$ for which
$W^{J\cap gHg^{-1}}$ is nonzero.
Let $I^\prime=\{\, \alpha=(i,j)\ | \ i\in I,\  1\le j\le n_i\,\}$.
Define spherical characters
$\sh_\pi^{(\alpha,\beta)}$, $\alpha$,
$\beta\in I^\prime$, as above.
Then $\{\, \sh_\pi^{(\alpha,\beta)}\ | \
\alpha,\, \beta\in I^\prime\,\}$ is a basis
for the space of spherical characters of
$\pi$.
\item If $\dim\Hom_H(\pi,1)=s<\infty$, then the 
dimension of the space of spherical
characters of $\pi$ is $s^2$.
\end{enumerate}
\end{theorem}

\begin{remark}\label{dimg} In \cite{HM}, we showed that
$\Hom_H(\pi,1)$ is finite-dimensional
for each of the irreducible supercuspidal
representations constructed in \cite{Y}.
Furthermore, we exhibit some examples
of such representations $\pi$ for which
$\dim\Hom_H(\pi,1)$ is greater than one.
\end{remark}

\section{Vanishing properties of spherical characters}
\label{sec:vanish}

Let $\pi=\Ind_J^G\kappa$ be as in the previous 
section. Throughout this section we assume
that the residual characteristic of $F$ is odd.
Let $G^{\tu}$ be the set of topologically unipotent
elements in $G$. Each element of $G^{\tu}$ belongs
to the pro-unipotent radical of some parahoric subgroup
of $G$.
Define $\tau:G\rightarrow G$  by $\tau(g)=g\,\theta(g)^{-1}$,
$g\in G$.

\begin{lemma}\label{trivint} 
Suppose that $\theta(J)=J$. Let $g$, $g^\prime\in G$ 
be such that
$\tau(g)$, $\tau(g^\prime)\in J$ and $J\,g\,H\not=J\,g^\prime H$.
Let $\U$ be an open neighbourhood of the identity
having the property that $\tau(\U)\subset G^{\tu}$.
Then  $\U\,\cap\, (Hg^{-1}J\,g^\prime H)=\emptyset$.
\end{lemma}

\begin{proof} Let $\dot J= g^{-1}Jg$. Then 
$$
\theta(\dot J)=\theta(g^{-1})\,\theta(J)\,\theta(g)=g^{-1}\tau(g)
\, \theta(J)\theta(\tau(g))g=\dot J,
$$
since $\tau(g)\in J$ and $\theta(J)=J$.
Note that $\tau(g^{-1}g^\prime)=g^{-1}\tau(g^\prime)\,
\theta(\tau(g))g\in g^{-1}Jg=\dot J$.
Furthermore, $H\, g\, J\, g^\prime H=
H\, \dot J\, g^{-1}g^\prime\, H$
and $JgH\cap Jg^\prime H=\emptyset$
is equivalent to $\dot J \, H\cap \dot J\, g^{-1}g^\prime H
=\emptyset$. Thus, after replacing $J$ by $\dot J$
and $g$ by $1$, if necessary, it suffices
to prove that $H \, J\, g^\prime H\cap \U=\emptyset$
whenever $g^\prime\notin JH$ and $\tau(g^\prime)\in J$.

Suppose that $h_1$, $h_2\in H$, $k\in J$. Let
$u=h_1k\, g^\prime h_2$. If $u\in \U$,
then $\tau(u)=h_1\tau(kg^\prime)h_1^{-1}\in \tau(\U)
\subset G^{\tu}$. Since $G^{\tu}$ is stable
under conjugation by elements of $G$, we
have $\tau(kg^\prime)\in G^{\tu}$.
Thus there exists a point $x$ in the Bruhat-Tits
building of $G$ and a positive real number
$t$ such that $\tau(kg^{\prime})\in G_{x,t}$, where
$G_{x,t}$ is the open compact subgroup of $G$
attached to the pair $(x,t)$ by Moy and Prasad (\cite{MP}).
Note that $\theta(\tau(kg^{\prime})=\tau(kg^{\prime})^{-1}$
and $\tau(kg^{\prime})=k\tau(g^\prime)\theta(k)^{-1}\in J^\prime$,
where
$J^\prime=J\cap G_{x,t}\cap\theta(G_{x,t})$
Applying Proposition~2.12 of \cite{HM} to the $\theta$-stable
group $J^\prime$, we have
$$
\tau(J^\prime)=\{\, k^\prime\in J^\prime\ | \ 
\tau(k^\prime)=k^{\prime -1}\,\}.
$$
Hence 
$\tau(kg^\prime)=\tau(k_1)$ for some $k_1\in J^\prime$.
This implies that $k_1^{-1}kg^\prime\in H$. As $k_1$, $k\in J$,
we have $g^\prime\in JH$, which is impossible.
Therefore $u\in H\, g^\prime J\, H$ implies $u\notin \U$.
\end{proof}

We remark that neighbourhoods $\U$ as in
the lemma are easy to find. For example,
if $K$ is a compact open subgroup of $G$ such
that $K\subset G^{\tu}$, we could
take $\U=K\cap\theta(K)$.

\begin{proposition}\label{van}
Let $\{\, g_i\ | \ i\in I\,\}$, $I^\prime$, and
$\{\, \sh_\pi^{(\alpha,\beta)}\ | \ \alpha,\, \beta\in I^\prime\,\}$
be as in Theorem~\ref{dimen}(1).
Assume that $\theta(J)=J$ and $\tau(g_i)\in J$ for
all $i\in I$. 
Let $\alpha=(i,j)$, $\beta=(\ell,m)\in I^\prime$.
Choose $\U$ as in Lemma~\ref{trivint}.
Then, if $i\not=\ell$,
$$
\sh_\pi^{(\alpha,\beta)}(g)=0\ \qquad \forall\ g\in \U\cap G^{\tr}.
$$
\end{proposition}

\begin{proof} Assume that $i\not=\ell$. It suffices
to show that $\sh_\pi(f)=0$ for all $f\in C_c^\infty(G)$
such that the support of $f$ lies inside $\U$.
Let $f$ be such a function. Since
$\sh_\pi^{(\alpha,\beta)}=D_{\varphi^{(\alpha,\beta)}}$,
we have
$$
\sh_\pi^{(\alpha,\beta)}(f)=\int_{H/H\cap Z}
\int_{H/H\cap Z}\int_G f(g)\,\langle\, w_j^{(i)*},
\, \F_{w_m^{(\ell)}}(g_ih_1gh_2g_\ell^{-1})\,\rangle_W
\,dg\,dh_1^\times dh_2^\times.
$$
According to Lemma~\ref{van}, 
if $g\in \U$ and $h_1$, $h_2\in H$,
then $g_ih_1gh_2g_\ell^{-1}\notin J$.
Because $\F_{w_m^{(\ell)}}$ is supported on $J$
and $f$ is supported in $\U$, this forces
$\sh_\pi^{(\alpha,\beta)}(f)=0$. 
\end{proof}

Next, suppose that $\G$ splits over
a tamely ramified extension of $F$.
Let $\pi$ be one of the irreducible supercuspidal
representations of $G$ constructed in \cite{Y}.
Each such $\pi$ is of the form $\Ind_J^G\kappa$
for some $J$ and $\kappa$.
The results of \cite{HM} give a description of
$\Hom_H(\pi,1)$ for $H=\G^\theta(F)$, subject to 
some hypotheses on quasicharacters of 
the $F$-rational points of certain reductive
$F$-subgroups of $\G$.
The hypotheses are satisfied under mild constraints
on the residual characteristic of $F$, and
they always hold when $G$ is a general linear group.
(The reader may refer to Section~2.6  of \cite{HM}
for a precise statement of the hypotheses.)
The following proposition is a restatement of some parts
of Theorem~5.26 of \cite{HM}. It follows
from the proposition that the vanishing results of
Proposition~\ref{van} hold for the spherical characters
of those supercuspidal representations of \cite{Y}
for which $\dim \Hom_H(\pi,1)$ is greater than one.

\begin{proposition} (\cite{HM}) Let $\pi$ be one of 
the supercuspidal
representations constructed in \cite{Y}. 
Assume that the hypotheses of \cite{HM} are satisfied.
Then
\begin{enumerate}
\item $\Hom_H(\pi,1)$ is finite-dimensional.
\item
If $\Hom_H(\pi,1)$ is nonzero then
$J$ may be chosen so that $\theta(J)=J$
and $\tau(g)\in J$ for all $g\in G$ such that 
$W^{J\cap gHg^{-1}}\not=\{\,0\,\}$.
\end{enumerate}
\end{proposition}

\begin{corollary} If $\dim\Hom_H(\pi,1)>1$,
then there exist spherical characters of
$\pi$ that vanish on the intersection of
$G^\tr$ with some open neighbourhood of
the identity.
\end{corollary}

\bibliographystyle{amsalpha}

\begin{thebibliography}{AKT}

\bibitem[AKT]{AKT} U.K.\ Anandavardhanan, A.C.\ Kable, and R.\ Tandon,
\textit{Distinguished representations and poles of twisted
tensor $L$-functions}, Proc.\ Amer.\ Math.\ Soc.\
\textbf{132} (2004), 2875--2883.

\bibitem[AR]{AR} U.K.\ Anandavardhanan and C.S.\ Rajan,
\textit{Distinguished representations, base change,
and reducibility for unitary groups}, Int.\ 
Math.\ Res.\ Not.\  (2005), no.\ 14,  841--854.

\bibitem[AT]{AT} U.K.\ Anandavardhanan and R.\ Tandon,
\textit{On distinguishedness}, Pacific J.\ Math.\
\textbf{206} (2002), 269--286.

\bibitem[BD]{BD} P.\ Blanc and P.\ Delorme, \textit{Vecteurs distributions
$H$-invariants de repr{\'e}sentations induites, pour un espace
sym{\'e}trique r{\'e}ductif $p$-adique} $G/H$, preprint.

\bibitem[BK]{BK} C.J.\ Bushnell and P.C.\ Kutzko,
\textit{The Admissible Dual of GL(N) via Compact Open Subgroups},
Annals of Mathematics Studies, Number 119, Princeton University Press (1993).

\bibitem[C]{C} L.\ Clozel, \textit{Characters of nonconnected, reductive
$p$-adic groups}, Canad.\ J.\ Math.\ \textbf{39} (1987), 149-167.

\bibitem[F]{F} Y.Z.\ Flicker, \textit{On distinguished representations},
J.\ Reine Angew.\ Math.\ \textbf{418} (1991), 139--172.

\bibitem[H1]{H1} J.\ Hakim, \textit{Distinguished $p$-adic representations},
Duke J.\ Math.\ \textbf{62} (1991), 1-22.

\bibitem[H2]{H2} J.\ Hakim, \textit{Character relations for
distinguished representations}, Amer.\ J.\ Math.\ \textbf{116}
(1994), 1153--1202.

\bibitem[H3]{H3} J.\ Hakim, \textit{Admissible distributions
on $p$-adic symmetric spaces}, J. reine angew.\ Math. \textbf{455}
(1994), 1-19.

\bibitem[H4]{H4} J.\ Hakim,
\textit{Supercuspidal Gelfand pairs}, J.~Number Theory 
\textbf{100} (2003),
251--269.

\bibitem [HMa]{HMa} J.\ Hakim and Z.\ Mao,
\textit{Cuspidal representations associated to $(GL(n),O(n))$ 
over finite fields and $p$-adic fields}, J.~Algebra \textbf{213} (1999),
129--143.

\bibitem [HM1]{HM1} J.\ Hakim and F.\ Murnaghan,
\textit{Tame supercuspidal representations of
$GL(n)$ distinguished by a unitary group},
Comp.\ Math. \textbf{133}
 (2002),  199--244.

\bibitem [HM2]{HM2} J.\ Hakim and F.\ Murnaghan,
\textit{Two types of distinguished supercuspidal
representations}, Int.\ Math. Res.\ Not.\ (2002)
no.~35, 1857--1889.

\bibitem [HM3]{HM} J. Hakim and F. Murnaghan, \textit{Distinguished
tame supercuspidal representations}, preprint.

\bibitem [HC1]{HC1} Harish-Chandra, \textit{Harmonic analysis
on reductive $p$-adic groups}, Lecture Notes in Math., vol.\ 162,
Springer-Verlag, Berlin (1970).

\bibitem [HC2]{HC2} Harish-Chandra, \textit{Admissible
distributions on reductive $p$-adic groups},
Preface and notes by Stephen DeBacker and Paul J.\ Sally Jr,
University Lecture Series, vol.\ 16 (American Mathematical
Society, Providence, RI, 1999).

\bibitem [HH]{HH} A.G.\ Helminck and G.F.\ Helminck,
\textit{A class of parabolic $k$-subgroups associated
with symmetric $k$-varieties}, Trans.\ Amer.\ Math.\
Soc.\ \textbf{350} (1998), 4669-4691.

\bibitem [HW]{HW} A.G.\ Helminck and S.P.\ Wang,
\textit{On rationality properties of involutions of
reductive groups}, Adv.\ Math. \textbf{99} (1993), 26--96.

\bibitem[Ka]{Ka} A.C.\ Kable, \textit{Asai $L$-functions and 
Jacquet's conjecture}, Amer.\ J.\ Math.\ \textbf{126} (2004), 789--820.

\bibitem[K]{K}
J.-L.~Kim, \textit{An exhaustion theorem: supercuspidal representations}, 
J.\ Amer.\ Math.\ Soc., to appear.

\bibitem [MP]{MP} A.\ Moy and G.\ Prasad, \textit{Unrefined minimal
$K$-types for $p$-adic groups}, Invent.\ Math.\ \textbf{116}
(1994), 393--408.

\bibitem[Pr1]{Pr1} D.\ Prasad, \textit{Distinguished representations
for quadratic extensions}, Comp.\ Math.\ \textbf{119} (1999),
335--345.

\bibitem[Pr2] {Pr2} D.\ Prasad, \textit{On a conjecture of
Jacquet about distinguished representations of $GL(n)$},
Duke Math.\ J.\ {\bf 109} (2001), 67--78.

\bibitem [RR]{RR} C.\ Rader and C.\ Rallis, \textit{Spherical 
characters on $p$-adic symmetric spaces}, Amer. J. Math.\ \textbf{118}
 (1996), 91-178.

\bibitem[RS]{RS} C.\ Rader and A.\ Silberger, \textit{Some consequences
of Harish-Chandra's submersion principle}, Proc.\ Amer.\ Math.\
Soc.\ \textbf{118} (1993), 1271--1279.

\bibitem [Y]{Y} J.-K.\ Yu, \textit{Construction of tame
supercuspidal representations}, J.\ Amer.\ Math.\
Soc.\ \textbf{14} (2001), 579--622.

\end{thebibliography}

\end{document}